\documentclass{article}
\usepackage {amssymb}
\usepackage {amsfonts}
\usepackage {amsmath}
\usepackage {amsthm}
\usepackage{graphicx}
\usepackage{graphics}
\usepackage {framed}
\usepackage {amsxtra}
\usepackage {enumerate}
\usepackage[margin=1in]{geometry}

\theoremstyle{plain}
\newtheorem{thm}{Theorem}[section]
\newtheorem{cor}[thm]{Corollary} 
\newtheorem{lem}[thm]{Lemma}

\theoremstyle{definition}
\newtheorem{defn}{Definition}

\theoremstyle{remark}

\begin{document}

\title{Rotation groups}

\author {
Donald Silberger\footnote{DonaldSilberger@gmail.com\ \ State University of New York\ \ New Paltz, NY 12561} 
\and Sylvia Silberger\footnote{sylvia.silberger@hofstra.edu\ \ Hofstra University\ \ Hempstead, NY 11549}}

\maketitle

\centerline{\sf Remembering Jacqueline Bare Grace --- 1942 December 01 - 2016 November 22}

\begin{abstract} A query, about the orbit $P{\cal W}$ in real 3-space of a point $P$ under an isometry group ${\cal W}$ generated by edge rotations of a tetrahedron, leads to contrasting notions, ${\cal W}$ 
versus ${\cal S}$, of ``rotation group''. The set R $=\{r_{{\sf A}_1},r_{{\sf A}_2}\}$ of rotations $r_{{\sf A} _i}$ about axes ${\sf A}_i$ generates two manifestations of an isometry group on $\Re^3$: 

(1). In the {\em stationary} group ${\cal S:=S}$(R), all axes {\sf B} are fixed under a rotation $r_{\sf A}$ about {\sf A}. 

(2). In the {\em peripatetic} group ${\cal W:=W}$(R), each $r_{\sf A}$ transforms every rotational axis ${\sf B\not=A}$.  
 
{\bf Theorem.} \ If the line ${\sf A}_1$ is skew to ${\sf A}_2$, if each $r_{{\sf A}_i}$ is of infinite order, and if $P\in\Re^3$, then both of the orbits $P{\cal S}$ and $P{\cal W}$ are dense in $\Re^3$. \end{abstract}

\section{Introduction}

Four decades ago, Jan Mycielski posed this question about a regular tetrahedron {\sf T}:\vspace{.3em}

{\sf For ${\cal G}$ the isometry group on 3-space generated by the edge rotations $r_{\sf E}$ of {\sf T}, where Size$(r_{\sf E})$ is the supplement of the dihedral angle of {\sf T}, what can be said 
about the orbit, $P{\cal G}:=\{Pf:f\in{\cal G}\}$, of a point $P$ affixed to {\sf T}? }\vspace{.3em}

Mycielski's response, to our recent answer to his question, led to our study of two manifestations, ${\cal W}$ and ${\cal S}$, of the rotation group generated by a set of rotations. We establish sufficient conditions 
for the orbits $P{\cal W}$ and $P{\cal S}$ of a point $P$ to be dense in $\Re^d$ for $d\in\{2,3\}$. We have not studied the case where $d\ge4$. 


\section{Technicalities} 

$\Re:=(-\infty,\infty)$ denotes the set of all real numbers,  ${\Bbb Z}$ is the set of all integers, and ${\Bbb N}:={\mathbb Z}^+$ is the set of all positive integers. When $n\in{\Bbb N}$ then 
$[n] := \{1,2,\ldots,n\}$. Finally, ${\mathbb Q}$ is the set of all rational numbers. 

We use standard interval notation; e.g.,  $(-3,7] := \{x:x\in\Re\wedge-3<x\le7\}$. Except where more particularly specified, the dimension $d\ge2$ of the real $d$-space $\Re^d$ of our isometries is arbitrary.

For $d=2$, the rotational axes are points in $\Re^2$. However, when $d\ge3$, the axes are directed lines in $\Re^d$. 

$r_{\sf A}:\Re^d\rightarrow \Re^d$ is a rotation about the axis {\sf A}, and is a {\em directed} isometry of $\Re^d$. By this we mean that a nonzero rotation has a sign. We deem $r_{\sf A}$ positive iff we see 
it as counterclockwise. For $d\ge3$, we judge $r_{\sf A}$ to be counterclockwise if we view it as counterclockwise when we look in the direction accorded to line {\sf A}. Let $F$ be a finite set, and 
R $:=\{r_{{\sf A}_i}:i\in F\}$ a set of {\em generator rotations}, with exactly one generator $r_{{\sf A}_i}$ per axis ${\sf A}_i$. For each $i\in F$, let ${\cal G}_i$ be the cyclic 
group $\{r_{{\sf A}_i}^z: z\in{\Bbb Z}\}$ of rotations about the axis ${\sf A}_i$. 

We write Rad$(r)=\rho\pi$ to indicate the radian measure, modulo $2\pi$, of the angle through which $r$ rotates; here $\rho\in(-1,1]$ unless otherwise specified. \ Size$(r)$ denotes $|$Rad$(r)|$. We call 
an angle $\angle PVQ$ rational iff Rad$(\angle PVQ)/\pi \in {\Bbb Q}$, and we call $r_{\sf A}$ rational if it rotates through a rational angle; i.e., if $\rho\in{\Bbb Q}$. Obviously $|{\cal G}_i| < \infty$ if and 
only if $r_{{\sf A}_i}$ is rational or, equivalently, is of finite order. Finally, we call two rotations {\em equal} iff they have both the same rotational axis and also the same radian measure modulo $2\pi$.

$f:P\mapsto Pf$ presents the isometry $f$ for $P\in\Re^d$. When ${\sf M}\subseteq\Re^d$, we write $f:{\sf M}\mapsto{\sf M}f := \{Vf:V\in{\sf M}\}$.  

The binary operation of each of the two sorts of groups in this paper is its own sort of left-to-right composition of isometries; that for a ``stationary'' group is the conventional $P(f\circ g) = (Pf)g$. But, as 
we will explain later, a more peculiar relationship, $P(f\star g)=(Pf)(fg)$, holds for a ``peripatetic'' group.

The expression  $f^-$ denotes the inverse of the isometry $f$, and so $r_i^-$ is the reverse-sense rotation of $r_i$. The identity isometry on $\Re^d$ is $\iota$.\vspace{.5em}

\subsection{Stationary groups}

Let R $:=\{r_i:i\in F\}$ for $2\le |F|<\infty$, where the axis of $r_i$ is ${\sf A}_i$ for $i\in F$. The following stipulates a specializing name for that entity which most would  take to be ``the'' 
R-generated rotational isometry group.  

\begin{defn} \label{Stationary}  For $\{f,h\}\subseteq{\cal S}$ and all $P\in\Re^d$, we let $P(f\circ h) = (Pf)h := (P(f\circ h_1))g_i$ where $h=h_1\circ g_i$ with $g_i\in{\cal G}_i$ for some  $i\in F$   
and some $h_1\in{\cal S}$. The {\em stationary} group is $\langle{\cal S},\circ\rangle$ where ${\cal S := S}$(R). \end{defn}

\noindent{\bf Stationary Example.} To illustrate the compositional algorithm of the group ${\cal S:=S}$(R), we will compute the stationary product $s := g_1\circ g_2\circ g'_1\circ g_3\in {\cal S}$ of the 
rotational sequence $\langle g_1,r_2,g'_1,g_3\rangle\in{\cal E}_1\times{\cal E}_2\times{\cal E}_1\times{\cal E}_3$. For $P\in\Re^d$, we will produce $s:P\mapsto Ps$ in four steps:

\underline{(1)}. We enact $g_1:P\mapsto Pg_1$ by rotating $P$ about the axis ${\sf A}_1$ with $g_1$. 

\underline{(2)}.  $g_2:Pg_1\mapsto (Pg_1)g_2 = P(g_1\circ g_2)$ is realized by rotating the point $Pg_1$ about ${\sf A}_2$ with $g_2$. 

\underline{(3)}. We realize $g'_1:P(g_1\circ g_2)\mapsto P(g_1\circ g_2\circ g'_1)$ by rotating $P(g_1\circ g_2)$ about ${\sf A}_1$ with $g'_1$. 

\underline{(4)}. Finally, $g_3$ rotates $P(g_1\circ g_2\circ g'_1)$ about ${\sf A}_3$ to reach $P(g_1\circ g_2\circ g'_1\circ g_3) = Ps$.\vspace{.5em}

\noindent{\bf Caveat.} In the stationary context, neither rotational axes, nor points comprising them, are moved by group actions. However, a mobile point $P$ may own the same coordinate address as an 
(immobile) axis point $U$; \ i.e., ``$P=U$'' in the address-sharing sense. But whereas $U$ is unmoved by $f$, we have $f:P\mapsto Pf\ne P$. That is, $P$ and $U$ are distinct entities at the same location 
in $\Re^d$.\vspace{.5em}

We call a product $g_{j_1}\circ g_{j_2}\circ\cdots\circ g_{j_k}$ of $g_{j_i}\in{\cal E}_{j_i}$ {\em reduced} iff $j_i\ne j_{i+1}$ for every $i<k$.  

\begin{lem}\label{StatFree} Let  $f\in{\cal S}$. Then there is exactly one reduced sequence ${\bf g} :=\langle g_{j_1},g_{j_2},\ldots,g_{j_k}\rangle$ in $\bigcup\{{\cal G}_i:i\in F\}$ for which 
$f = g_{j_1}\circ g_{j_2}\circ\cdots\circ g_{j_k}$. So ${\cal S}$ is a free group on the generator set {\rm R}, modulo for those $i\in F$ with $|{\cal G}_i|<\infty$, to congruences $\mod2\pi$ which 
select the $|{\cal G}_i|$ representatives in $(-\pi,\pi]\cap{\cal G}_i$ from $\{r_i^z:z\in{\Bbb Z}\}$.  \end{lem}

\begin{proof} The lemma holds for $k=1$. Suppose for all $i<k$ that, if $h\in{\cal S}$ is a product  $h=g_{j_1}\circ g_{j_2}\circ\cdots\circ g_{j_i}$ of a reduced sequence in $\bigcup\{{\cal G}_t:t\in F\}$, then 
this factorization of $h$ is unique. Let $h := g_{j_1}\circ g_{j_2}\circ\cdots\circ g_{j_{k-1}}$, and let $f := h\circ g_{j_k}$. Pretend that $f = h\circ g_t$ for some $g_t \ne g_{j_k}$. Then, since both $f$ and $h$ 
are bijective transformations of $\Re^d$, we must  infer that $g_t = h^-\circ f = g_{j_k}$, a contradiction. The lemma follows.\end{proof}

\subsection{Peripatetic groups}

\begin{defn}\label{Peripatetic} For $g_i\in{\cal G}_i$ and  $f:\Re^d\rightarrow\Re^d$ an isometry, $fg_i$ denotes the rotation about the axis ${\sf A}_if$ with Rad$(fg_i) =$ Rad$(g_i)$. The binary operation 
$\star$ of the group ${\cal W := W}$(R) is defined recursively by $P(f\star g_i) := (Pf)(fg_i)$ for all $P\in\Re^d$, and by $\iota\star f=f=f\star\iota$. We call the group 
$\langle{\cal W}$(R)$;\star\rangle$ {\em peripatetic}.\end{defn}

${\cal W}$ is called ``peripatetic'' as a reminder that ${\sf A}_jg_i\not={\sf A}_j$ for all $\{i,j\}\subseteq F$ with $i\ne j$ unless $g_i=\iota$.\vspace{.5em} 

\noindent{\bf Peripatetic Example.} Here $\langle g_1,g_2,g'_1,g_3\rangle$ is the same four-term sequence we employed in the Stationary Example above. Again let $P\in\Re^d$. We will illustrate the peripatetic 
group's computational algorithm by showing how to obtain $w:P\mapsto Pw$ via the peripatetic group product $w := g_1\star g_2\star g'_1\star g_3$.\vspace{.3em}  

Our calculation realizing $w:P\mapsto Pw$ proceeds in the following four steps. 

\underline{(1)}.  $g_1:P\mapsto Pg_1$ swings $\Re^d$ about axis ${\sf A}_1$, while rotating the axis ${\sf A}_2$ into the position ${\sf A}_2g_1$, and ${\sf A}_3$ into the position ${\sf A}_3g_1$; \  i.e., 
when ${\cal W}$ is the group, then $g_1:{\sf A}_i\mapsto{\sf A}_ig_1$ if $i\in\{2,3\}$.  Of course ${\sf A}_1g_1={\sf A}_1$. 

\underline{(2)}. The composite rotation $g_1g_2$ moves $Pg_1$ to $(Pg_1)(g_1g_2) = P(g_1\star g_2)$, by rotating $Pg_1$ about the new axis ${\sf A}_2g_1$. Simultaneously, 
$g_1g_2:\langle {\sf A}_1g_1,{\sf A}_2g_1,{\sf A}_3g_1\rangle\mapsto\langle {\sf A}_1(g_1\star g_2),{\sf A}_2(g_1\star g_2),{\sf A}_3(g_1\star g_2)\rangle$. 

\underline{(3)}. The composite rotation $(g_1\star g_2)g_1'$ rotates the point $P(g_1\star g_2)$ about the axis ${\sf A}_1(g_1\star g_2)$, thus effecting  
$P(g_1\star g_2)\mapsto (P(g_1\star g_2))g_1' = P((g_1\star g_2)\star g_1') = P(g_1\star g_2\star g_1')$. Concomitantly  moving axes, we get  
$(g_1\star g_2)g_1':\langle {\sf A}_1(g_1\star g_2),{\sf A}_2(g_1\star g_2),{\sf A}_3(g_1\star g_2)\rangle\mapsto 
\langle {\sf A}_1(g_1\star g_2\star g_1'),{\sf A}_2(g_1\star g_2 \star g_1'),{\sf A}_3(g_1\star g_2\star g_1')\rangle$. 

\underline{(4)}.  $(g_1\star g_2\star g_1')g_3:P(g_1\star g_2\star g_1')\mapsto P(g_1\star g_2\star g_1'\star g_3)$. Thus does the group ${\cal W}$(R) produce $w:P\mapsto Pw$.\vspace{.3em} 

We ask our readers to note the difference between this $Pw$ and the $Ps$ in the stationary example.\vspace{.5em}  

The tumbling {\sf T} issue is more naturally treated by ${\cal W}$ than by ${\cal S}$, since by design ${\cal W}$ maintains the axes of rotation inside {\sf T}.  Overall, ${\cal W}$ groups may lend themselves 
more readily to navigational strategies in a spaceship than do ${\cal S}$ groups, since a rotational coordinate system rooted in ${\cal W}$(R) travels with the traveler.\vspace{.5em}

We now firm up the foundation on which rest plausible but unproven assumptions about ${\cal W}$.  

\begin{lem}\label{StarFree} Let  $f\in{\cal W}$. Then there is exactly one reduced sequence ${\bf g} :=\langle g_{j_1},g_{j_2},\ldots,g_{j_k}\rangle$ in $\bigcup\{{\cal G}_i:i\in F\}$ for which 
$f = g_{j_1}\star g_{j_2}\star\cdots\star g_{j_k}$. Therefore ${\cal W}$ is a free group on the generator set {\rm R}, subject to the congruences modulo $2\pi$ of $|{\cal G}_i|$ for those $i\in F$ with 
$|{\cal G}_i|<\infty$.   \end{lem}

\begin{proof} Uniformly substituting ``$\star$'' for ``$\circ$'' in the proof of Lemma \ref{StatFree}, we obtain a proof of Lemma \ref{StarFree}.  \end{proof}

\begin{cor}\label{Isomorphic} The groups ${\cal S}${\rm(R)} and ${\cal W}${\rm (R)} are isomorphic. \end{cor}

\begin{proof} This corollary is immediate from the two Lemmas, \ref{StatFree} and \ref{StarFree}.  \end{proof}

Although as abstract algebraic structures the groups ${\cal S}$(R) and ${\cal W}$(R) are indistinguishable, they differ as specific subgroups of the group of all isometries on $\Re^d$. Given $P\in\Re^d$ and R, 
we ask our readers for an efficient algorithm that expresses some isomorphism $\varphi:{\cal S}$(R)$\rightarrow{\cal W}$(R). 

For $P\in\Re^d$ what relationships obtain between the sets $P{\cal S}$ and $P{\cal W}$?

\section{The tumbling tetrahedron and peripatetic orbital densities}

When $x\in\Re$ then the expression $x{\Bbb Z}/{\Bbb Z}$ denotes the set $[0,1)\cap\{xa-b:\{a,b\}\subseteq{\Bbb Z}\}$. 

Our applications of ${\cal W}$(R) to density issues depend upon the following well-known consequence\footnote{pointed out to us by Wies{\l}aw Dziobiak} of a theorem of Kronecker, \cite{apostol}. A theorem of 
Hurwitz, \cite{hardy} \cite{hurwitz} \cite{leveque} \cite{niven1}, also provides a proof.\footnote{We will email a PDF of three very short proofs of this fact to those who request us to do so.} 

\begin{lem}\label{Tool} Let $x\in\Re\setminus{\mathbb Q}$. Then the set $x{\mathbb Z}/{\mathbb Z}$ is dense in $[0,1)$.  \end{lem} 

We need a corollary which is immediate from Lemma \ref{Tool}:

\begin{cor}\label{Circle} Let $P$ be a point on a circle {\sf C} with centerpoint $U$, and let $r$ be a rotation about $U$. Then these three assertions are equivalent: 

{\bf 1.} \ The rotation $r$ is of infinite order.

{\bf 2.}  \ The size of $r$ is an irrational multiple of $\pi$ radians.

{\bf 3.}  \ The set $\{Pr^z:z\in{\mathbb Z}\}$ is dense in {\sf C}.\end{cor}

The generating sets R := $\{r_{{\sf A}_1},r_{{\sf A}_2}\}$ we will be using are two-membered, except when we deal with the tumbling tetrahedron {\sf T}, for which $|$R$_{\sf T}|=6$. Under the action of 
${\cal W}$, a copy of $\Re^d$ moves against a fixed $\Re^d$ background, whose points are not budged by ${\cal W}$ actions.

Recall that in both the stationary and the peripatetic contexts, ${\cal G}_i$ is the cyclic subgroup $\{r_i^z:i\in{\Bbb Z}\}$.

We now state and prove our more easily visualized peripatetic rotational density theorem.\footnote{Allan Silberger suggested Theorem \ref{AllanS}. } 

\begin{thm}\label{AllanS} Let $U_1\not=U_2$ be points in $\Re^2$. Let $r_1$ and $r_2$ be rotations of infinite order about $U_1$ and $U_2$ respectively. Let ${\cal W:=W}${\rm(R)}, 
where {\rm R} $:= \{r_1,r_2\}$. Let $P\in\Re^2$. Then the orbit $P{\cal W}$ is dense in $\Re^2$.    \end{thm}

\begin{proof} Take it that $P\ne U_1$. We argue by contradiction. Pretend there is a (fixed) point $Q$, and a (fixed) disc ${\sf D}\subseteq\Re^2$ with center at $Q$, and whose radius $\epsilon>0$ is the 
largest that allows ${\sf D}\cap P{\cal W}=\emptyset$. Since $\epsilon\le\|P-Q\|<\infty$, we have $Pf\in{\sf S}:=\overline{\sf D}\setminus{\sf D}$ for some $f\in{\cal W}$, where $\overline{\sf D}$ is the 
closure of {\sf D}, and {\sf S} is a circle. 

The points $Pf,\, (U_2f)(fg_1) =: U_2(f\star g_1)$ and $Q$ are noncollinear for some $g_1\in{\cal G}_1$. Let ${\sf C}_1\subseteq\Re^2$ be the circle with centerpoint $U_1$ and with $P\in{\sf C}_1$; by 
hypothesis, the radius of ${\sf C}_1$ is $\|P-U_1\|>0$. Let ${\sf C}_2$ be the circle with centerpoint $U_2$ and with $P\in{\sf C}_2$; the radius of ${\sf C}_2$ is $\|P-U_2\|\ge 0$. 

The circles ${\sf C}_1$ and ${\sf C}_2$ are mapped by $f$ isometrically onto the respective circles ${\sf C}_1f$ with centerpoint $U_1f$, and ${\sf C}_2f$ with centerpoint $U_2f$. We see that 
${\sf C}_1f = {\sf C}_1(f\star g_1)$, since $f\star g_1$ merely rotates ${\sf C}_1f$ about the point $U_1f$. However, ${\sf C}_2(f\star g_1)\not={\sf C}_2f$; for, the rotation $f\star g_1$ swings $\Re^2$ 
around $U_1f$, thus mapping the circle ${\sf C}_2f$ isometrically onto the circle ${\sf C}_2(f\star g_1)$, whose centerpoint is $U_2fg_1\not=U_2f$. If $g_1$ is chosen prudently, then 
${\sf C}_2(f\star g_1)\cap{\sf B}\not=\emptyset$. Now \ref{Circle} finishes the proof, since $P(f\star g_1)=(Pf)(fg_1)$ and $f\star g_1\in{\cal W}$. \end{proof}

The analogous next result will enable us to answer the Mycielski question which inspired this paper.

\begin{thm}\label{Main} Let $r_1$ and $r_2$ be infinite-order rotations about the respective skew directed lines ${\sf X}_1$ and ${\sf X}_2$ in $\Re^3$, let {\rm R} $:=\{r_1,r_2\}$, let ${\cal W:=W}${\rm (R)}, and 
let $P\in\Re^3$. The orbit $P{\cal W} := \{Pf:f\in{\cal W}\}$ is dense in $\Re^3$.  \end{thm}

\begin{proof} We can take it that $P\notin{\sf X}_1$. For each $i\in\{1,2\}$ with $P\notin{\sf X}_i$, let $U_i\in{\sf X}_i$ be such that $PU_i\perp{\sf X}_i$, and let ${\sf C}_i$ be the circle in $\Re^3$ of radius 
$\|P-U_i\|$ and with centerpoint $U_i$. If $P\notin{\sf X}_i$, then $\{U_i,P,Pg_i\}$ is a set of noncollinear points for $\iota \ne g_1\in{\cal G}_1$. But in the event that $P\in{\sf X}_i$, let $U_i:=P$ and 
${\sf C}_i := \{P\}$. 

Again arguing by contradiction, we pretend that there is a (fixed) point $Q$, and a $Q$-centered open ball ${\sf B}\subseteq\Re^3$ of radius $\epsilon>0$, such that ${\sf B}\cap P{\cal W}=\emptyset$, 
but that if ${\sf B'}$ is a $Q$-centered ball with radius $\epsilon'>\epsilon$ then ${\sf B'}\cap P{\cal W}\not=\emptyset$. Since $\epsilon \le \|Q-P\| <\infty$, it follows that 
$Pf\in{\sf S}:=\overline{\sf B}\setminus{\sf B}$ for some $f\in{\cal W}$,  where $\overline{\sf B}$ is the closure of {\sf B} and where {\sf S} is therefore a 2-sphere. 

Since {\sf B} is open, if ${\sf B\cap C}_1f\not=\emptyset$ then ${\sf B\cap C}_1f$ is an arc of positive length in {\sf B}, whence Corollary \ref{Circle} concludes our proof. So we take the circle ${\sf C}_1f$ to be 
tangent to the sphere {\sf S} at the point $Pf$. For $i\in\{1,2\}$, the point $U_if\in{\sf X}_if$ is the centerpoint of the circle ${\sf C}_if$.

Since the axes ${\sf X}_1$ and ${\sf X}_2$ are skew and since $f$ is an isometry, ${\sf X}_1f$ is skew to ${\sf X}_2f$. Furthermore, Order$(r_2) = \infty$ by hypothesis. WeThere are two cases. 

\underline{Case}: $P\notin{\sf X}_2$. Then we can infer from Corollary \ref{Circle} that there exists $g_2\in{\cal G}_2$ with $P(f\star g_2) := (Pf)(fg_2)\in{\sf B}$. Moreover, $f\star g_2 \in {\cal W}$. Thus 
${\sf B}\cap P{\cal W} \ne \emptyset$, contrary to our choice of ${\sf B}$. 

\underline{Case}: $P=U_2$. By Corollary \ref{Circle} there exists $g_2\in{\cal G}_2$ with ${\sf B}\cap{\sf C}_1(f\star g_2) \ne \emptyset$. So here too ${\sf B}\cap P{\cal W} \ne \emptyset$.

\noindent In both cases we reach a contradiction. \end{proof}

\noindent{\bf Remark.} If the hypotheses of Theorem \ref{Main} allowed ${\sf X}_1\cap{\sf X}_2=\{V\}$, then $P{\cal W}\subseteq{\sf Y}$ where ${\sf Y}\subseteq\Re^3$ is the sphere of radius $\|P-V\|$ and 
centerpoint $V$; if also $P\not= V$ then $\|P-V\|>0$, and $P{\cal W}$ would be dense in {\sf Y}.\vspace{.7em} 

In order that assure that $P{\cal W}({\rm R}_{\sf T})$ is dense in $\Re^3$, where \ R$_{\sf T}$ \ is the set of six edge rotations $r_{\sf E}$ of {\sf T}, Theorem \ref{Main} requires only two of the six to be irrational.
But  Size$(r_{\sf E}) := \pi-\theta$ for each $r_{\sf E}\in {\rm R}_{\sf T}$, where $\theta$ is the size of each dihedral angle of {\sf T}. So we will need to prove that $\theta$ is irrational.\vspace{.5em}

\noindent{\bf Quiz.} Show that $\sin(\theta)=2\sqrt{2}/3$, and that equivalently $\cos{\theta} = 1/3$.

\begin{lem}\label{Niven} Let $0<\theta<\pi/2$. If $\cos(\theta)=1/3$ then $\theta$ is irrational. \end{lem}

\begin{proof} By Corollary 3.12 of \cite{niven1}, both the angle $\pi/2-\phi$ and the real number $\cos(\phi)=\sin(\pi/2-\phi)$ are rational if and only if $\phi = \pi/3$. So, since $\cos(\pi/3)=1/2 \not= 
1/3 = \cos(\theta)$, we see that $\theta$ is irrational. \end{proof}

The answer to Mycielski's half-century-old query about tumbling {\sf T} is now obvious from \ref{Main} with \ref{Niven}.

\begin{cor}\label{Tetrahed} Let {\sf T} be a regular tetrahedron in $\Re^3$, and let ${\cal W := W}({\rm R}_{\sf T})$ be the peripatetic rotational isometry group determined by the set ${\rm R}_{\sf T}$ of six 
generating rotations $r_{\sf E}$ around the edges {\sf E} of {\sf T}, where each {\rm Size}$(r_{\sf E})$ is the supplement of the dihedral angle of {\sf T}. Then the orbit $P{\cal W}$ is dense in $\Re^3$ for each $P\in\Re^3$.   \end{cor}

Steve Silverman asks whether there exists a tetrahedron, all six of whose dihedral angles are rational. He provides an example in which four of the six are rational. Is four best possible? If ``yes'', then the tumblings 
of an arbitrary tetrahedron {\sf K}, each edge {\sf E} of which is assigned a rotation $r_{\sf E}$ whose size is the supplement of the dihedral angle at {\sf E}, will trace out an orbit $P{\cal W}$ that is dense in 
$\Re^3$ for each $P\in\Re^3$.

\section{Orbital density for stationary groups}

The following may be a duplication of a decades-old unpublished result of Jan Mycielski.

\begin{thm}\label{StaDense} Let $d\in\{2,3\}$. Let the generating set {\rm R} be as in Theorem \ref{AllanS} or \ref{Main}, and let both of the rotations $r_1$ and $r_2$ be of infinite order. Let 
${\cal S:=S}${\rm(R)} be the stationary group determined  {\rm R}. Then the orbit $V{\cal S}$ is dense in $\Re^d$ for every $V\in\Re^d$. \end{thm} 

\begin{proof} Since the proof for $d=3$ is essentially the same as that for $d=2$, we will argue only the $d=2$ case.   

The axes of $r_1$ and $r_2$ are $U_1$ and $U_2$ respectively. However, unlike in Theorem \ref{AllanS}, here the points $U_i$ are fixed, and are unaffected by actions of ${\cal S}$. On the other hand, 
the point $V$ is not fixed. Suppose without loss of generality that $V\not= U_2$; \ i.e., that the points $V$ and $U_2$ have different addresses in $\Re^2$. 

By hypothesis, for each $i\in\{1,2\}$ we have that Size$(r_i) = \rho_i\pi$ for an irrational $\rho_i$. We can take it that $\rho_i\in(0,1)$. Let $k_i$ be the positive integer for which $k_i\rho_i< 1 < (k_i+1)\rho_i$. 

Each set ${\sf C}^*_{i,V} := {\sf C}_{i,V}\cap V{\cal S}$ is dense in the unique circle ${\sf C}_{i,V}$ with centerpoint $U_i$ and with $V\in{\sf C}_i$, provided only that ${\sf C}_i\not=\{U_i\}$. Let ${\sf O}_1$ 
be the finite set  $\{Vr_1^j: 0\le j\le 2k_1\}\subset{\sf C}_{1,V}^*$. 

For each $A\in{\sf O}_1$, we create the finite set $\{Ar_2^j:0\le j\le 4k_2\}\subset {\sf C}_{2,A}^*$, where ${\sf C}_{2,A}$ is the unique circle having centerpoint $U_2$ and such that $A\in{\sf C}_{2,A}$.  
Define ${\sf O}_2 :=\bigcup\{\{Ar_2^j:0\le j\le 4k_2\}:A\in{\sf O}_1\}$. 

Similarly, using the points $A\in{\sf O}_2$, define ${\sf O}_3 := \bigcup\{\{Ar_1^j:0\le j\le 2^3k_1\}:A\in{\sf O}_2\}$. 

Next, define ${\sf O}_4 := \bigcup\{\{Ar_2^j:0\le j\le 2^4k_2\}:A\in{\sf O}_3\}$. Continue thus, back and forth between the fixed rotational axes $U_1$ and $U_2$. This recursion produces an infinite sequence 
$\langle{\sf O}_b\rangle_{b=1}^\infty$ of finite subsets of $V{\cal S}$ with ${\sf O}_b\subset{\sf O}_{b+1}$, where eventually the convex closures of the ${\sf O}_b$ approximate elliptical regions whose minimal diameters 
increase without bound, and such that the meshes of the sets ${\sf O}_b$ diminish to zero.\footnote{It is natural to define the mesh of a subset ${\sf E}\subseteq V{\cal S}$ as the minimum $\epsilon>0$ for which no 
disc ${\sf D}$ of radius $\epsilon$ inside the convex closure of {\sf E} can fail to contain points in {\sf E}. }  Ultimately we obtain a denumerable subset $\bigcup\{{\sf O}_b:b\in{\mathbb N}\}\subseteq V{\cal S}$, 
where $\bigcup\{{\sf O}_b:b\in{\mathbb N}\}$ is dense in $\Re^2$. Therefore, the orbit $V{\cal S}$ of $V$ under the actions of the stationary rotational isometry group ${\cal S}$(R) is itself dense in $\Re^2$. \vspace{.5em} 

\noindent The three-dimensional case involves a reiteration of the argument above, modulo these five observations:

(1). The axes of rotation are now two fixed directed skew lines instead of two fixed points. 

(2). The circle ${\sf C}_{1,V}$ is not coplanar with the circle ${\sf C}_{2,V}$.

(3). For $d=3$, the convex closures of the ${\sf O}_b$, approximate ellipsoids or spheroids rather than ellipses. 

(4). Re ``mesh'': In $\Re^3$ we talk of $\epsilon$-radius balls instead of the $\epsilon$-radius discs that make sense in $\Re^2$. 

(5). In $\Re^2$ the convex closure of $V{\cal S}$ is two-dimensional. But it is three-dimensional in $\Re^3$.  \end{proof}

\section{Related issues}

The classes of stationary and peripatetic rotational groups in $\Re^d$ for $2\le d<\infty$ are antipodal in regards to the modes of their rotational generators $r_i$. This suggests the possible fruitfulness of studying 
rotational groups of sorts that lie between those antipodal classes; i.e., rotational groups some of whose given axes are stationary while others are peripatetic.\vspace{1em}

We believe Theorem \ref{Main} extends in a natural way to $\Re^d$ for $d\in\{4,5,6\ldots\}$, given an appropriate finite collection ${\cal X}$ of lines ${\sf X}_i\subseteq\Re^d$ and of infinite-order rotations $r_i$ 
of 2-planes about those lines. Our readers will expect that the claim in such a generalization of Theorem \ref{Main} would be that the orbit $P{\cal W}$ is dense in $\Re^d$ for every $P\in \Re^d$.

We ask this question of our readers: Which, if any, of the following three hypotheses about the collection ${\cal X}$ is both necessary and sufficient to guarantee the density in $\Re^d$ of $P{\cal W}$ for all $P$?

{\bf 1.} \ $|{\cal X}|=d-1$ and no proper hyperplane in $\Re^d$ is a superset of $\bigcup{\cal X}$. 

{\bf 2.} \ $|{\cal X}|=d-1$ and ${\cal X}$ is pairwise skew. 

{\bf 3.} \ ${\cal X}=\{{\sf X}_1,{\sf X}_2\}$ is a set of skew lines.\vspace{.5em} 

To illuminate the contention between the hypotheses, {\bf 1} and {\bf 2}, we offer for $4\le d\in{\mathbb N}$ this example:

For each $i\in{\mathbb N}$, let ${\sf Y}_i:=\{\langle i,ti,ti^2,0,0,\ldots 0\rangle:t\in\Re\}\subseteq\Re^d$. Then ${\cal Y}:=\{{\sf Y}_i:i\in{\mathbb N}\}$ is an infinite pairwise skew set of lines for which $\bigcup{\cal Y}$ 
is a subset of a copy of  $\Re^3$ that occurs as a proper subspace of $\Re^d$.\vspace{1em} 

Both Theorem \ref{Main} and Corollary \ref{Tetrahed} provide orbits $P{\cal W}$, each of which is dense in $\Re^3$. But the orbit encountered in \ref{Tetrahed} is plainly richer than that encountered in \ref{Main}. This 
invites an observation. 

Let the regular tetrahedron {\sf T} have an edge length of $\sqrt{6}$, let $P$ be its barycenter, let ${\cal G}$ be the peripatetic rotational isometry group generated by the six edge rotations $r_{UV}$ of {\sf T}, let 
$f\in{\cal G}$, and let the axes of rotations $r_{AB}$ and $r_{AC}$ share a single vertex, $A$. Then ${\sf H}\cap P{\cal G}$ contains the vertices of a tiling of {\sf H} by hexagons of edge length $1$, where {\sf H} is the 
plane determined by the point set $\{Pf,Pfr_{AB},Pfr_{AC}\}$.\vspace{.5em} 

\noindent{\bf Conjecture 1.} If plane ${\sf H}\supseteq\{Pf,Pfr_{AB},Pfr_{AC}\}$ then ${\sf H}\cap P{\cal G}$ is the vertex set of a hexagonal tiling of {\sf H}.\vspace{1em}

Let {\sf X} be the $x$-axis of $\Re^3$. Let {\sf Y} be parallel to the $y$-axis and through $\langle 0,0,1\rangle$.  Let ${\cal W := W}$(R) be the peripatetic group generated by ${\rm R} :=\{r_{\sf X},r_{\sf Y}\}$, 
 where Order$(r_{\sf X})=$ Order$(r_{\sf Y})=4$.  Let $P\in\Re^3\setminus({\sf X\cup Y})$. Then $P{\cal W}$ is an infinite discrete set of vertices of rectangular parallelopipeds, and $P{\cal W}$  is nowhere dense.

This suggests a shift of focus from rotations to their axes of rotation.\vspace{.5em}

For $d\ge 2$, say that nonparallel lines {\sf X} and {\sf Y} in $\Re^d$ {\em conform rationally} iff some translate of {\sf Y} intersects {\sf X} to form a rational angle.   
 
Let ${\sf X}_1$ and ${\sf X}_2$ be skew lines in $\Re^3$, and let both of their respective rotations $r_1$ and $r_2$ be of finite order. Let ${\cal F}$ be the peripatetic rotational isometry group generated by 
F $:=\{r_1,r_2\}$. \vspace{.5em}

\noindent{\bf Conjecture 2.} \ $P{\cal F}$ is nowhere dense for each $P\in\Re^3$ if and only if ${\sf X}_1$ conforms rationally to ${\sf X}_2$.\vspace{1em} 

\noindent{\bf Conjecture 3.} \ If ${\sf X}_1$ does not conform rationally to ${\sf X}_2$, then $P{\cal F}$ is dense in $\Re^3$ for all $P\in\Re^3\setminus({\sf X}_1\cup{\sf X}_2)$. \vspace{1em}

Suppose Conjecture 2 is true. Let ${\sf X}_1$ and ${\sf X}_2$ be rationally conforming skew lines in $\Re^3$, and let both $r_1$ and $r_2$ be rotations of finite order on those respective axes. Let $F := 
\{r_1,r_2\}$ be the generator set for the peripatetic group ${\cal F:=F}$(F).\vspace{.5em}

\noindent We call a polyhedron ${\sf K}\subseteq\Re^3$ a $\langle P,$F$\rangle$-{\em chamber} iff {\sf K} is maximal with respect to these four conditions:\vspace{.3em} 

(1). \ {\sf K} is convex and of finite diameter.

(2). \ Every vertex of {\sf K} is an element in $P{\cal F}$. 

(3). \ No element in $P{\cal F}$ is in the interior of {\sf K}.

(4). \ No three vertices of {\sf K} are collinear.\vspace{.3em}

\noindent What can one say about these $\langle P,$F$\rangle$-chambers? 




\vspace{3em} 

\noindent{\bf\Large Special language:} arc-rational, conform rationally, peripatetic, stationary. 


\begin{thebibliography}{99}                    


\bibitem{apostol} Tom M. Apostol. ``Modular Functions and Dirichlet Series in Number Theory,'' 2nd Edition. Springer-Verlag, New York, 1990. Chapter VII, page 148. 

\bibitem{hardy} G. H. Hardy, Edward M. Wright, Roger Heath-Brown, Joseph Silverman, Andrew Wiles (2008). "Theorem 193". An introduction to the Theory of Numbers (6th ed.). Oxford science publications. p. 209. ISBN 0-19-921986-9.

\bibitem{hurwitz} Adolf Hurwitz. (1891). \"{U}ber die angen\"{a}herte Darstellung der Irrationalzahlen durch rationale Br\"{u}che (On the approximate representation of irrational numbers by rational fractions)". Mathematische Annalen (in German). 39 (2): 279–284. doi:10.1007/BF01206656. JFM 23.0222.02. (A PDF version of the paper is available from the given weblink for the volume 39 of the journal, provided by G\"{o}ttinger Digitalisierungszentrum)

\bibitem{leveque} William Judson LeVeque (1956). ``Topics in number theory". Addison-Wesley Publishing Co., Inc., Reading, Mass. MR 0080682.

\bibitem{niven1} Ivan Niven (1956). ``Irrational numbers''. Carus Mathematical Monographs, No. 11. Math. Assoc. America. Distributed by John Wiley and Sons, Inc., New York, NY.

\bibitem{niven2} Ivan Niven (2013). Diophantine Approximations. Courier Corporation. ISBN 0486462676.


\end{thebibliography}
\end{document}